\documentclass[11pt]{amsart}

\newif\ifdraft\draftfalse

\usepackage{amssymb}
\usepackage{amsthm}
\usepackage{eucal}
\usepackage[english]{babel}
\usepackage{nicefrac}
\usepackage{times}
\usepackage{latexsym,amscd,amsmath,amsthm,amssymb}

\makeatletter
\def\@begintheorem#1#2[#3]{%
    \def\naam{#1}
  \deferred@thm@head{\the\thm@headfont \thm@indent
    \@ifempty{#1}{\let\thmname\@gobble}{\let\thmname\@iden}%
    \@ifempty{#2}{\let\thmnumber\@gobble}{\let\thmnumber\@iden}%
    \@ifempty{#3}{\let\thmnote\@gobble}{\let\thmnote\@iden}%
    \thm@swap\swappedhead\thmhead{#1}{#2}{#3}%
    \the\thm@headpunct
    \thmheadnl 
    \hskip\thm@headsep
  }%
  \ignorespaces}
\makeatother

\newcommand{\kantlijndraft}[1]{\ifdraft\hspace{-\lastskip}%
\vadjust{\vspace{-1mm}\smash{\llap{{\tt #1}\hspace{8mm}}}\vspace{1mm}}\fi}

\def\voegToe#1#2#3{\immediate\write1{\string\newlabel{#1}{{#2}{#3}}}}

\newcommand{\thlabel}[1]{\voegToe{#1}{\naam\noexpand~\thetheorem}{\thepage}\kantlijndraft{#1}}

\makeatletter
\renewcommand{\label}[1]{\voegToe{#1}{\@currentlabel}{\thepage}\kantlijndraft{#1}}
\makeatother

\newtheorem{theorem}{Theorem}[section]

\newtheorem{corollary}[theorem]{Corollary}
\newtheorem{question}[theorem]{Question}

\newtheorem{proposition}[theorem]{Proposition}
\theoremstyle{definition}

\theoremstyle{remark}

\numberwithin{equation}{section}

\newtheorem{claim2}{\sc Claim}



\newcommand{\pichar}[1]{\ensuremath{\pi\chi(#1)}}

\newcommand{\density}[1]{\ensuremath{d(#1)}}

\newcommand{\sse}{\subseteq}						
\newcommand{\minus}{\backslash}						
\newcommand{\Un}{\bigcup}							
\newcommand{\Meet}{\bigcap}							
\newcommand{\meet}{\cap}							
\newcommand{\es}{\varnothing}						
\newcommand{\reals}{\mathbb{R}}

\newcommand{\cl}[1]{\ensuremath{\overline{#1}}}

\newcommand{\scr}[1]{\ensuremath{\mathcal{#1}}}

\def\cprime{$'$}
\def\cont{\mathfrak{c}}
\def\sapirovskii{{\v{S}}apirovski{\u\i}}
\def\arhangelskii{Arhangel{\cprime}ski{\u\i}}

\def\juhasz{Juh{\'a}sz}

\begin{document}

\title{A survey of cardinality bounds on homogeneous topological spaces}

\author{Nathan Carlson}\address{Department of Mathematics, California Lutheran University, 60 W. Olsen Rd, MC 3750, Thousand Oaks, CA 91360 USA}
\email{ncarlson@callutheran.edu}



\begin{abstract} 
In this survey we catalogue the many results of the past several decades concerning bounds on the cardinality of a topological space with homogeneous or homogeneous-like properties. These results include van Douwen's Theorem, which states $|X|\leq 2^{\pi w(X)}$ if $X$ is a power homogeneous Hausdorff space \cite{VD78}, and its improvements $|X|\leq d(X)^{\pi\chi(X)}$ \cite{Rid06} and $|X|\leq 2^{c(X)\pi\chi(X)}$ \cite{CR08} for spaces $X$ with the same properties. We also discuss de la Vega's Theorem, which states that $|X|\leq 2^{t(X)}$ if $X$ is a homogeneous compactum \cite{DLV2006}, as well as its recent improvements and generalizations to other settings. This reference document also includes a table of strongest known cardinality bounds on spaces with homogeneous-like properties. The author has chosen to give some proofs if they exhibit typical or fundamental proof techniques. Finally, a few new results are given, notably (1) $|X|\leq d(X)^{\pi n\chi(X)}$ if $X$ is homogeneous and Hausdorff, and (2) $|X|\leq \pi\chi(X)^{c(X)q\psi(X)}$ if $X$ is a regular homogeneous space. The invariant $\pi n\chi(X)$, defined in this paper, has the property $\pi n\chi(X)\leq\pi\chi(X)$ and thus (1) improves the bound $d(X)^{\pi\chi(X)}$ for homogeneous Hausdorff spaces. The invariant $q\psi(X)$, defined in \cite{ism81}, has the properties $q\psi(X)\leq\pi\chi(X)$ and $q\psi(X)\leq\psi_c(X)$ if $X$ is Hausdorff, thus (2) improves the bound $2^{c(X)\pi\chi(X)}$ in the regular, homogeneous setting.
\end{abstract}

\maketitle

\section{Introduction}
A topological space $X$ is \emph{homogeneous} if for every $x,y\in X$ there exists a homeomorphism $h:X\to X$ such that $h(x)=y$. Roughly, $X$ is homogeneous if the topology at every point is ``identical'' to that of every other point. $X$ is \emph{power homogeneous} if there exists a cardinal $\kappa$ such that $X^\kappa$ is homogeneous. Many commonly studied spaces are homogeneous (for example, $\mathbb{R}^2$, the unit circle, all connected manifolds in general, and topological groups) and as such are ubiquitous across fields of mathematics. In particular, those homogeneous spaces that are compact play a prominent role. In 1931 Keller \cite{Keller} showed that the Hilbert Cube $[0,1]^\omega$ is homogeneous. As $[0,1]$ is not homogeneous, this was an early example of a compact power homogeneous space that is not homogenous. Another such example is the ordinal space $\omega+1$, as $(\omega+1)^\omega$ is homogeneous. The author refers the reader to the 2014 book chapter \emph{Topological Homogeneity} by A.V.~\arhangelskii~and J. van Mill ~\cite{av14} for a broad and extensive reference on the theory of general homogeneous topological spaces. 

\arhangelskii~\cite{arh69} showed in 1969 that the cardinality of any compact, first countable space is at most $\cont$, the cardinality of the continuum $\reals$, thus answering a 50 year old question of Alexandroff and Urysohn. Soon afterwards, in 1970, he showed that any compact, homogeneous, sequential space also has cardinality at most $\cont$ \cite{arh70a}\cite{arh70b}. This demonstrated that in the presence of homogeneity the first countable condition can be relaxed to the weaker sequential condition. This might be regarded as the first example of cardinality bound that can be improved if a space is additionally known to be homogeneous. In the decades that followed, and in recent years, many well-known cardinality bounds on topological spaces have been improved with homogeneity, or homogeneous-like properties. 

The purpose of this survey is to give a thorough account of subsequent results concerning the cardinality of a homogeneous topological space. While there are important open problems in the more general theory of cardinal functions on homogeneous spaces (such as van Douwen's Problem, which asks if the cellularity $c(X)$ of a homogeneous compactum is at most $\mathfrak{c}$), in this survey we confine ourselves only to cardinality considerations.

Several proofs are given in this survey. The ones that are chosen were chosen for their illustrative nature, as they are fundamental to the theory of cardinality bounds on homogeneous spaces. They were also chosen for their simplicity and elegance, in the author's opinion. Theorems that have proofs that are more involved and complicated are simply cited in this survey. In addition, there are a few new proofs given in this paper that represent mostly minor improvements of known results. 

This paper is organized as follows. In \S2 we consider van Douwen's Theorem, a groundbreaking 1978 result that established the first cardinality bound for a general homogeneous Hausdorff space. In \S3 and \S4, we explore two bounds for homogeneous Hausdorff spaces that improved van Douwen's Theorem: $d(X)^{\pi\chi(X)}$ and $2^{c(X)\pi\chi(X)}$, respectively. In \S5 we consider compact homogeneous spaces, de la Vega's Theorem, and improvements to that theorem, while in \S6 we look at the many extensions and generalizations of de la Vega's Theorem to the Hausdorff and other settings. In \S7 we consider other results not in these categories, in \S8 we compile a list of questions and give a table of cardinality bounds for homogeneous-like spaces that are strongest, as known to the author.

In this survey we do not assume any additional separation axioms on a topological space. All topological spaces are, in fact, topological spaces. We refer the reader to \cite{EN} and \cite{Juhasz} for all undefined terms.

\section{van Douwen's Theorem}

In 1978 Eric van Douwen \cite{VD78} showed that $|X|\leq 2^{\pi w(X)}$ for a Hausdorff homogeneous space $X$. This was first cardinality bound for a general homogeneous space $X$. Homogeneity, or homogeneous-like properties are necessary in this result; for example, the non-homogeneous space $\beta\omega$ does not satisfy this bound. In fact, van Douwen showed that this bound holds for general power homogeneous spaces using sophisticated ``clustering'' techniques that encode information about projection maps of the form $\pi: X^\kappa\to X$. Van Douwen made extensive use of collections of sets invariant under homeomorphisms. As one can see from the next theorem, his paper was primarily focused on results that imply a space is \emph{not} homogeneous, or not power homogeneous. His cardinality bound was indeed simply just a consequence of this sophisticated theorem.

\begin{theorem}[van Douwen \cite{VD78}, 1978]\label{VDdeep}
If the space $X$ admits a continuous map onto a Hausdorff space $Y$ with $|Y|>2^{\pi w(Y)}$, then no power of $X$ is homogeneous in each of the following cases:
\begin{itemize}
\item[(a)] $f$ is open or is a retraction and $d(X)\leq \pi w(Y)$,
\item[(b)] $f$ is perfect, $X$ is regular, and $d(X)\leq \pi w(Y)$, or
\item[(c)] $X$ is compact Hausdorff and $w(X)\leq 2^{\pi w(Y)}$.
\end{itemize}
\end{theorem}
As $d(X)\leq\pi w(X)$ for any space $X$, we have the following corollary, which we will refer to as van Douwen's Theorem.
\begin{corollary}[van Douwen \cite{VD78}, 1978]\label{VD}
If $X$ is power homogeneous then $|X|\leq 2^{\pi w(X)}$. 
\end{corollary}

In fact, it also follows from Theorem \ref{VDdeep} that no power of $\beta\omega\minus\omega$ is homogeneous, answering a question of Murray Bell. (Frolik \cite{Fro} had previously shown this space is not homogeneous in ZFC). 

While van Douwen's work was groundbreaking and answered important questions, it turns out that proofs that $|X|\leq 2^{\pi w(X)}$ when $X$ is \emph{homogeneous} are, by comparison, straightforward. In this survey we'll see several proofs that imply this result (see Theorems~\ref{aut}, \ref{density}, and \ref{ER}). Furthermore, the full version of van Douwen's Theorem, in the power homogeneous case, has been improved upon in different directions in subsequent decades. See Theorems~\ref{ridHausdorff}, \ref{open}, and \ref{ERPH}.

Corollary~\ref{VD} has a straightforward improvement in the homogeneous case by considering the group $H(X)$ of autohomeomorphisms on a space $X$. This was shown by Frankiewicz. The proof we give here is adapted from that given in \cite{Juhasz}, 2.38. The author considers this proof to involve basic techniques that are typically used when considering bounds on the cardinality of a homogeneous space and the group $H(X)$. 

\begin{theorem}[Frankiewicz \cite{Fra79}, 1979]\label{aut}
If $X$ is Hausdorff then $|H(X)|\leq 2^{\pi w(X)}$.
\end{theorem}

\begin{proof}
Let $\kappa =\pi w(X)$ and let $\scr{B}$ be a $\pi$-base for $X$ such that $|\scr{B}|\leq\kappa$. We show that map $\phi:H(X)\to\scr{P}(\scr{B})^\scr{B}$ defined by $\phi(h)(B)=\{C\in\scr{B}:C\sse h[B]\}$ is one-to-one.  Suppose we have $f,g\in H(X)$ such that $f\neq g$. Then there exists $x\in X$ such that $f(x)\neq g(x)$. Let $U$ and $V$ be disjoint neighborhoods of $f(x)$ and $g(x)$, respectively. As $x\in f^{\leftarrow}[U]\meet g^{\leftarrow}[V]$, there exists $B\in\scr{B}$ such that $B\sse f^{\leftarrow}[U]\meet g^{\leftarrow}[V]$. Thus $f[B]\sse U$ and $g[B]\sse V$. There exists $C\in\scr{B}$ such that $C\sse f[B]$, and thus $C$ is not a subset of $g[B]$. This shows $\phi(f)(B)\neq\phi(g)(B)$ and $\phi(f)\neq\phi(g)$. We conclude that $\phi$ is one-to-one and therefore $|H(X)|\leq|\scr{P}(\scr{B})^\scr{B}|\leq (2^\kappa)^\kappa=2^\kappa$.
\end{proof}

As it is easily seen that $|X|\leq |H(X)|$ if $X$ is homogeneous, the homogeneous case of \ref{VD} follows. Years later, in 2008, an improved bound for $|H(X)|$ was given by the author and Ridderbos using the Erd\"os-Rado theorem from partition theory.

\begin{theorem}[C., Ridderbos \cite{CR08}, 2008]\label{CRH}
If $X$ is Hausdorff then $|H(X)|\leq 2^{c(X)\pi\chi(X)sd(X)}$.
\end{theorem}

The \emph{separation degree} $sd(X)$ in the cardinal inequality above is defined as follows. We say that a subset $Z$ of $X$ \emph{separates} a subset $\scr{G}$ of $H(X)$, if for all $f,g\in\scr{G}$ with $f\not= g$ there is some $z\in Z$ with $f(z)\not= g(z)$. $sd(X)$ is defined by
$sd(X) = \min\{ |Z| : Z~\mbox{separates}~H(X)\}$. It is always the case that $sd(X)\leq\density{X}$, and thus \ref{CRH} is a logical improvement of \ref{aut}. 

\section{The cardinality bound $d(X)^{\pi\chi(X)}$}
The proof of Theorem~\ref{aut} gives a straightforward way to demonstrate van Douwen's Theorem in the homogeneous case. A few years later, in 1981, Ismail~\cite{ism81} gave another relatively simple proof with a slightly stronger conclusion that used the notion of a $q$-pseudo base. For a point $x$ in a space $X$, a family $\scr{B}$ of nonempty open subsets of $X$ is a $q$\emph{-pseudo base} of $x$ in $X$ if for each $y\in X$ such that $y\neq x$, there is a subfamily $\scr{C}$ of $\scr{B}$ such that $x\in\overline{\Un\scr{C}}$ and $y\notin\overline{\Un\scr{C}}$. Ismail defined the \emph{$q$-pseudo character} of $x$ in $X$ by $q\psi(x,X)=\min\{|\scr{B}|:\scr{B}\textup{ is a }q\textup{-pseudo base of }x\textup{ in }X\}$ and the $q$-pseudo character of $X$ by $q\psi(X)=\sup\{q\psi(x,X):x\in X\}$. It was shown in \cite{ism81} that if $X$ is Hausdorff then $q\psi(X)\leq |X|$, $q\psi(X)\leq \pi\chi(X)$, and that $q\psi(X)\leq\psi_c(X)$. 

Recall a set $U$ in a space $X$ is regular open if $U=intclU$ and that $RO(X)$ denotes the collection of regular open subsets of $X$. Ismail showed the following fundamental result, for which we provide a proof. Integral to the proof is the fact that in a homogeneous space $X$ if one fixes a point $p\in X$ there exist homeomorphisms $h_x:X\to X$ such that $h_x(p)=x$ for each $x\in X$. These homeomorphisms play an important role in most proofs of cardinality bounds on homogeneous spaces. The proof exhibits how these homeomorphisms interact with the invariant family $RO(X)$. Upon examination, the proof does not require the Hausdorff property, despite the fact that Ismail listed that property as an hypothesis. 

\begin{theorem}[Ismail \cite{ism81}, 1981]\label{ismailq}
If $X$ is a homogeneous space, then $|X|\leq |RO(X)|^{q\psi(X)}$.
\end{theorem}

\begin{proof}
Fix a point $p\in X$ and let $\scr{B}$ be a q-pseudo base at $x$ in $X$ such that $|\scr{B}|=q\psi(X)$. Without loss of generality we can assume that $\scr{B}\sse RO(X)$, for otherwise we could consider the q-pseudo base $\{int\overline{B}:B\in\scr{B}\}\sse RO(X)$, which has the same cardinality as $\scr{B}$. For all $x\in X$, there exists a homeomorphism $h_x:X\to X$ such that $h_x(p)=x$. 

Define a function $\phi:X\to RO(X)^\scr{B}$ by $\phi(x)(B)=h_x[B]$. (Note that if $B\in RO(X)$ then $h_x[B]\in RO(X)$). We show $\phi$ is one-to-one. Let $x,y\in X$ such that $x\neq y$. Then $h_x^{\leftarrow}(y)\neq p$. As $\scr{B}$ is a q-pseudo base at $p$, there exists $\scr{C}\sse\scr{B}$ such that $p\in\overline{\Un\scr{C}}$ and $h_x^{\leftarrow}(y)\notin \overline{\Un\scr{C}}$. Therefore, $y=h_y(p)\in\overline{\Un\{h_y[C]:C\in\scr{C}\}}$ and $y\notin\overline{\Un\{h_x[C]:C\in\scr{C}\}}$. It follows that there exists a $C\in\scr{C}\sse\scr{B}$ such that $h_x[C]\neq h_y[C]$, and thus $\phi(x)(C)\neq \phi(y)(C)$ and $\phi(x)\neq\phi(y)$. This shows $\phi$ is one-to-one and $|X|\leq\left|RO(X)^{\scr{B}}\right|\leq |RO(X)|^{q\psi(X)}$.
\end{proof}

To see that the above result is logically stronger than Theorem~\ref{VD} in the homogeneous case, recall that $|RO(X)|\leq 2^{d(X)}$ for an arbitrary space $X$ (see, for example, 2.6d in \cite{Juhasz}) and that $\pi w(X) = d(X)\pi\chi(X)\geq d(X)q\psi(X)$ if $X$ is Hausdorff.

\begin{corollary}[Ismail \cite{ism81}, 1981]\label{ismail}
If $X$ is a homogeneous Hausdorff space, then $|X|\leq |RO(X)|^{\pi\chi(X)}$.
\end{corollary}

It was noted independently by de la Vega~\cite[Theorem 1.14]{DLVthesis} and Ridderbos~\cite[Proposition 2.2.7]{RidMasters} that $|RO(X)|$ can be replaced in \ref{ismail} by the invariant $d(X)$. The author considers the proof of this result to be elegant, representative, and a fundamental model for more sophisticated related cardinality bounds on spaces with homogeneous-like properties. One might consider this result the analogue of the cardinality bound $|X|\leq d(X)^{\chi(X)}$ for a general Hausdorff space $X$ and, in fact, involves a simpler one-to-one map argument using the homeomorphisms $h_x$. We give this proof below and the reader should view it as fundamental in the theory of cardinality bounds on homogeneous spaces.

\begin{theorem}[de la Vega~\cite{DLVthesis}, 2005, and Ridderbos~\cite{RidMasters}]\label{density}
If $X$ is homogeneous and Hausdorff then $|X|\leq d(X)^{\pichar{X}}$.
\end{theorem}
\begin{proof}
Fix a point $p\in X$ and a local $\pi$-base $\scr{B}$ at $p$ consisting of non-empty open sets such that $|\scr{B}|\leq\pichar{X}$. Let $D$ be a dense subset of $X$ such that $|D|=d(X)$. For all $x\in X$ let $h_x:X\to X$ be a homeomorphism such that $h_x(p)=x$. For all $x\in X$ and $B\in\scr{B}$, $h_x[B]$ is a non-empty open set and thus there exists $d(x,B)\in h_x[B]\meet D$. Define $\phi:X\to D^{\scr{B}}$ by $\phi(x)(B)=d(x,B)$. 

We show $\phi$ is one-to-one. Let $x\neq y\in X$. Separate $x$ and $y$ by disjoint open sets $U$ and $V$, respectively.Then $p\in h_x^{\leftarrow}[U]\meet h_y^{\leftarrow}[V]$, an open set. There exists $B\in\scr{B}$ such that $B\sse h_x^{\leftarrow}[U]\meet h_y^{\leftarrow}[V]$. It follows that $\phi(x)(B)=d(x,B)\in h_x[B]\sse U$ and $\phi(y)(B)=d(y,B)\in h_y[B]\sse V$. Thus $\phi(x)(B)\neq\phi(y)(B)$ and $\phi(x)\neq\phi(y)$. This shows $\phi$ is one-to-one and $|X|\leq |D|^{|\scr{B}|}\leq d(X)^{\pichar{X}}$.
\end{proof}

It was shown in \cite{car07} that the semiregularization $X_s$ of a space $X$ is homogeneous if $X$ is homogeneous. (See, for example, \cite{por88} for a thorough discussion of the semiregularlization of a space). Using results in \cite{car07}, Theorem~\ref{density}, and the fact that $d(X_s)\leq RO(X)$, we have $|X| = |X_s|\leq d(x_s)^{\pi\chi(X_s)}\leq RO(X)^{\pi\chi(X)}$. Therefore Ismail's result~\ref{ismail} above follows from~\ref{density}. However, Theorem~\ref{ismailq} and Theorem~\ref{density} appear to be incomparable.

Theorem~\ref{density} has a minor but interesting improvement by replacing $\pi\chi(X)$ with a smaller cardinal function the author will call $\pi n\chi(X)$. We define $\pi n\chi(X)$ as follows. For a point $x$ in a space $X$, we define a \emph{local $\pi$-network} at $x$ to be a collection $\scr{N}$ of sets (not necessarily open) such that if $x\in U$ and $U$ is open, then there exists $N\in\scr{N}$ such that $N\in U$. Denote by $\pi n\chi(x,X)$ the least infinite cardinal $\kappa$ such that $x$ has a local $\pi$-network $\scr{N}$ of cardinality $\kappa$ and $\chi(N,X)\leq\kappa$ for all $N\in\scr{N}$. Define $\pi n\chi(X)=\sup\{\pi n\chi(x,X):x\in X\}$. Observe that $\pi n\chi(X)\leq\pi\chi(X)$. 

The following result appears to be new in the literature. The proof also involves the construction of a one-to-one map, however there is another ``layer'' in this construction above and beyond what is done in the proof of Theorem~\ref{density}. 

\begin{theorem}\label{densityn}
If $X$ is homogeneous and Hausdorff then $|X|\leq d(X)^{\pi n\chi(X)}$.
\end{theorem}

\begin{proof}
Let $\kappa = \pi n\chi(X)$. As in the proof of Theorem~\ref{density}, fix $p\in X$ and for every $x\in X$ fix a homeomorphism $h_x:X\to X$ such that $h_x(p)=x$. There exists a local $\pi$-network $\scr{N}$ at $p$ such that $|\scr{N}|\leq\kappa$ and $\chi(N,X)\leq\kappa$ for all $N\in\scr{N}$. Let $D$ be dense in $X$ such that $|D|=d(X)$. Let $\{U(N,\alpha):\alpha<\kappa\}$ be a neighborhood base at $N$ for each $N\in\scr{N}$. 

As $D$ is dense, for all $x\in X$, for all $N\in\scr{N}$, and for all $\alpha<\kappa$, there exists a point $d(x,N,\alpha)\in h_x[U(N,\alpha)]\meet D$. We define a function $\phi: X\to (D^\kappa)^\scr{N}$ by $\phi(x)(N)(\alpha)=d(x,N,\alpha)$ and show $\phi$ is one-to-one. Let $x\neq y\in X$ and separate $x$ and $y$ by disjoint open sets $U$ and $V$, respectively. Then $p\in h_x^\leftarrow[U]\meet h_y^\leftarrow[V]$. As $\scr{N}$ is a local $\pi$-network for $p$, there exists $N\in\scr{N}$ such that $N\sse h_x^\leftarrow[U]\meet h_y^\leftarrow[V]$. As $\{U(N,\alpha):\alpha<\kappa\}$ is a neighborhood base at $N$, there exists $\alpha<\kappa$ such that $U(N,\alpha)\sse h_x^\leftarrow[U]\meet h_y^\leftarrow[V]$. Thus, $h_x[U(N,\alpha)]\sse U$ and $h_y[U(N,\alpha)]\sse V$, showing $d(x,N,\alpha)\neq d(y,N,\alpha)$. It follows that $\phi(x)(N)(\alpha)\neq\phi(y)(N)(\alpha)$, $\phi(x)(N)\neq\phi(y)(N)$, and finally that $\phi(x)\neq \phi(y)$. This shows $\phi$ is one-to-one, and 
$$|X|\leq\left|(D^\kappa)^\scr{N}\right|\leq (|D|^\kappa)^\kappa=|D|^\kappa\leq d(X)^{\pi n\chi(X)}.$$
\end{proof}

We turn now to the setting in which a space $X$ is power homogeneous and not necessarily homogeneous. In the full power homogeneous setting, van Douwen's Theorem~\ref{VD} has been improved in variety of ways using differing techniques. In the study of power homogeneous spaces $X$, information on the homogeneity of $X^\kappa$ for a cardinal $\kappa$, and the autohomeomorphisms on that space, must be utilized in some way at the level of the space $X$. This information must be captured in such a way as to generate inequalities involving cardinal functions on $X$. The projection maps $\pi:X^\kappa\to X$ are typically, and necessarily, used in this process. In 2006, Ridderbos~\cite{Rid06} used new techniques involving projection maps to give the first improvement to \ref{VD} in the full power homogeneous setting. 

\begin{theorem}[Ridderbos \cite{Rid06}, 2006]\label{ridHausdorff}
If $X$ is a power homogeneous Hausdorff space, then $|X|\leq d(X)^{\pi\chi(X)}$.
\end{theorem}

If $X^\kappa$ is homogeneous then, as in any homogeneous space, after fixing a point $p\in X^\kappa$, there exist homeomorphisms $h_x:X^\kappa\to X^\kappa$ such that $h_x(p)=x$ for every $x\in X^\kappa$. However, a critical ingredient in the proof of Theorem~\ref{ridHausdorff} is demonstrating the existence of such homeomophisms with additional important properties relating to local $\pi$-bases in $X^\kappa$ and $X$. This is shown in Corollary 3.3 in \cite{Rid06}.

Recently, Bella and the author extended \ref{ridHausdorff} to give a bound for the cardinality of any open set in a power homogeneous Hausdorff space. 

\begin{theorem}[Bella, C. \cite{BC}, 2018]\label{open}
Let X be a power homogeneous space. If $D\sse X$ and $U$ is an open set such that $U\sse\overline{D}$, then $|U|\leq |D|^{\pi\chi(X)}$.
\end{theorem}

Recall that a subset $D$ of a space $X$ is $\theta$-dense in $X$ if $\overline{U}\meet D\neq\es$ for every non-empty open set $U$ of $X$. The $\theta$-\emph{density} of $X$ is defined by $d_\theta(X)=\textup{min}\{|D|:D\textup{ is }\theta\textup{-dense in }X\}$. A variation of the proof of \ref{ridHausdorff} in the Urysohn setting was given by the author in \cite{car07}.

\begin{theorem}[C. \cite{car07}, 2007]\label{Ury}
If $X$ is power homogeneous and Urysohn then $|X|\leq d_\theta(X)^{\pichar{X}}$.
\end{theorem}

\section{The cardinality bound $2^{c(X)\pi\chi(X)}$}

Van Douwen's Theorem~\ref{VD} also has an improvement in a different direction. Theorem~\ref{density}, coupled with the fact that $d(X)\leq\pi\chi(X)^{c(X)}$ for any regular space $X$ (\sapirovskii~\cite{Sap1974}), shows that $|X|\leq 2^{c(X)\pi\chi(X)}$ for regular homogeneous spaces $X$. This was observed by \arhangelskii~in \cite{arh87}. In~\cite{car07}, these results we modified to show $|X|\leq 2^{c(X)\pi\chi(X)}$ for Urysohn homogeneous spaces $X$. This follows from Theorem~\ref{Ury} and the fact that $d_\theta(X)\leq \pi\chi(X)^{c(X)}$ for any space $X$ \cite{car07}.

As $c(X)\pi\chi(X)\leq \pi w(X)$ for any space, we see that $2^{c(X)\pi\chi(X)}$ is an improved bound over $2^{\pi w(X)}$. It is a real improvement, even in the compact case, as the compact right topological group $X$, constructed under CH by Kunen~\cite{Kunen}, satisfies $c(X)\pi\chi(X) = \omega$ and $\pi w(X) = \omega_1$. 

The question remained open whether this bound was valid in the Hausdorff case. Using entirely different techniques, Ridderbos and the author answered this in the affirmative in~\cite{CR08}. While relatively simple, it represented the first use of the Erd\"os-Rado Theorem in the proof of a cardinal inequality involving homogeneous spaces. It is related to the proof of the Hajnal-\juhasz~theorem $|X|\leq 2^{c(X)\chi(X)}$ for general Hausdorff spaces that uses the Erd\"os-Rado theorem (see~\cite[2.15b]{Juhasz}). Indeed, one may view this result as the homogeneous analogue of the Hajnal-\juhasz~theorem. We give this proof below as our next fundamental proof.

\begin{theorem}[C., Ridderbos \cite{CR08}, 2008]\label{ER}
If $X$ is homogeneous and Hausdorff then $|X|\leq 2^{c(X)\pichar{X}}$.
\end{theorem}

\begin{proof}
Let $\kappa=c(X)\pichar{X}$. Fix a point $p\in X$ and a local $\pi$-base $\scr{B}$ at $p$ consisting of non-empty sets such that $|\scr{B}|\leq\kappa$. For all $x\in X$ let $h_x:X\to X$ be a homeomorphism such that $h_x(p)=x$. Define $B:[X]^2\to\scr{B}$ as follows. For all $x\neq y$, there exist disjoint open sets $U(x,y)$ and $V(x,y)$ containing $x$ and $y$, respectively. For each $x\neq y\in X$ the open set $h_x^{\leftarrow}[U]\meet h_y^{\leftarrow}[V]$ contains $p$. Thus there exists $B(x,y)\in\scr{B}$ such that $B(x,y)\sse h_x^{\leftarrow}[U]\meet h_y^{\leftarrow}[V]$. Note that $h_x[B(x,y)]\meet h_y[B(x,y)]=\es$ for all $\{x,y\}\in [X]^2$.

By way of contradiction suppose that $|X|>2^\kappa$. By the Erd\"os-Rado Theorem there exists $Y\in[X]^{\kappa^+}$ and $B\in\scr{B}$ such that $B=B(x,y)$ for all $x\neq y\in Y$. For $x\neq y\in Y$ we have $h_x[B]\meet h_y[B]=h_x[B(x,y)]\meet h_y[B(x,y)]=\es$. This shows that $\scr{C}=\{h_x[B]:x\in Y\}$ is a cellular family. But $|\scr{C}|=|Y|=\kappa^+>c(X)$, a contradiction. Therefore $|X|\leq 2^\kappa$.
\end{proof}

Using Ismail's invariant $q\psi(X)$, one observes that the above theorem has a improvement in the case when the space $X$ is additionally regular. The proof is a simple matter of lining up a few results, although the result appears to be new in the literature.
\begin{theorem} \label{regular}
If $X$ is regular and homogeneous, then $|X|\leq\pi\chi(X)^{c(X)q\psi(X)}$.
\end{theorem} 
\begin{proof} 
Since $X$ is regular and homogeneous, we have
$$|X|\leq |RO(X)|^{q\psi(X)}\leq (\pi\chi(X)^{c(X)})^{q\psi(X)}\leq\pi\chi(X)^{c(X)q\psi(X)}.$$
The first inequality above is Theorem~\ref{ismailq}, and the second inequality follows from the inequality $|RO(X)|\leq\pi\chi(X)^{c(X)}$ for regular spaces (see~\cite{Juhasz} 2.37).
\end{proof}

This is an actual improvement over the bound $2^{c(X)\pichar{X}}$ because $q\psi(X)\leq\pi\chi(X)$ for a Hausdorff space $X$. Furthermore, it improves the cardinality bound $\pi\chi(X)^{c(X)\psi(X)}$ for regular spaces $X$ given by \sapirovskii~\cite{Sap1974}, as $q\psi(X)\leq \psi(X)$ if $X$ is regular. One may view Theorem~\ref{regular} as the homogeneous analogue of \sapirovskii's result.

We turn now to the case where the space $X$ is power homogeneous. In this case, van Mill \cite{VM05} first demonstrated the bound $2^{c(X)\pi\chi(X)}$ holds under the assumption of compactness, using a variation of van Douwen's clustering techniques.

\begin{theorem}[van Mill \cite{VM05}, 2005]\label{vMcompact}
If $X$ is a power homogeneous compactum, then $|X|\leq 2^{c(X)\pi\chi(X)}$.
\end{theorem}

Van Mill's result follows in fact as a corollary to this result in the same paper: $|X|\leq w(X)^{\pi\chi(X)}$ for a power homogeneous compactum $X$. (Recall that later it was shown that $|X|\leq d(X)^{\pi\chi(X)}$ for any power homogeneous Hausdorff space by Ridderbos (Theorem~\ref{ridHausdorff})). We will see, however, in Theorem~\ref{ERPH} below that the cardinality bound $2^{c(X)\pi\chi(X)}$ holds for any power homogeneous Hausdorff space. 

Soon after van Mill's result, Bella gave an improvement of Theorem~\ref{VD} for regular power homogenous spaces using the cardinal function $c^*(X)$. Recall that if $X=\prod_{i\in T}X_i$ is an arbitrary product of spaces and $c(\prod_{i\in F}X_i)\leq\kappa$ for each finite subset $F$ of $T$, then $c(X)\leq\kappa$. If follows that if $X$ is a space and $\lambda$ is an infinite cardinal then $c(X^\lambda)=c^*(X)$.

\begin{theorem}[Bella \cite{bel05}, 2005]
If $X$ is a power homogeneous $T_3$ space, then $|X|\leq 2^{c^*(X)\pi\chi(X)}$.
\end{theorem}

Note that $c^*(X)\pi\chi(X)\leq \pi w(X)$ for any space. Bella's result is also a real improvement of Theorem~\ref{VD}, as the same space $X$ in \cite{Kunen} satisfies $c^*(X)\pi\chi(X) = \omega$ and $\pi w(X) = \omega_1$.

In 2008 it was finally shown that $2^{c(X)\pi\chi(X)}$ is a bound for the cardinality of any power homogeneous Hausdorff space. This represents a second full improvement of van Douwen's theorem alongside Theorem~\ref{ridHausdorff}. The proof of this is a sophisticated application of the Erd\"os-Rado theorem.

\begin{theorem}[C. and Ridderbos \cite{CR08}, 2008]\label{ERPH}
If $X$ is power homogeneous and Hausdorff then $|X|\leq 2^{c(X)\pichar{X}}$.
\end{theorem}

A decade later a variation of this result was shown for spaces that are Urysohn or quasiregular. Recall that a space is \emph{quasiregular} if every nonempty open set contains a nonempty regular closed set. A collection of nonempty open sets is a \emph{Urysohn cellular family} if the closures of any two are disjoint. We define the \emph{Urysohn cellularity} of a space $X$ as $Uc(X)=\sup\{|\scr{C}|: \scr{C}\textup{ is a Urysohn cellular family}\}$. It is clear that $Uc(X)\leq c(X)$ for any space $X$.

\begin{theorem}[Bonanzinga, C., Cuzzup\'e, Stavrova \cite{BCCS2018}, 2018]\label{phUry}
If $X$ is a power homogeneous space that is Urysohn or quasiregular, then $|X|\leq 2^{Uc(X)\pi\chi(X)}$.
\end{theorem}

\section{Compact homogeneous spaces and de la Vega's Theorem}

In 1970 ~\arhangelskii~ showed that the cardinality of a sequential, homogeneous compactum is at most $\cont$~\cite{arh70a}\cite{arh70b}. (By \emph{compactum} we mean a compact Hausdorff space). He then asked if the sequential property can be relaxed to countably tight (see~\cite{arh70a}). In 2006, de la Vega~\cite{DLV2006} answered this long-standing question by showing that the cardinality of a homogeneous compactum is bounded by $2^{t(X)}$. Previously Dow~\cite{Dow88} had shown under PFA that any compact space $X$ of countable tightness contains a point of countable character; thus if the space is additionally homogeneous then $|X|\leq\mathfrak{c}$. 

De la Vega's original proof involved the elementary submodels technique and, in fact, showed that $|X|\leq 2^{L(X)t(X)pct(X)}$ for any regular homogeneous space. (See the next section for the definition of $pct(X)$ and a discussion of this and other generalizations of de la Vega's Theorem). It was observed in \cite{CR2012} that much of the work of de la Vega's elementary submodel proof can be replaced by a theorem of Pytkeev concerning covers by $G_\kappa$-sets. If $X$ is a space and $\kappa$ an infinite cardinal, the $G_\kappa$-\emph{modification} $X_\kappa$ of $X$ is the space formed on the underlying set $X$ by taking the collection of $G_\kappa$-sets as a basis.

\begin{theorem}[Pytkeev \cite{Pyt1985}, 1985]\label{pyt}
Let $X$ be a compactum and $\kappa$ an infinite cardinal. Then $L(X_\kappa)\leq 2^{t(X)\cdot\kappa}$.
\end{theorem}

Another crucial ingredient in the proof of de la Vega's theorem is a result of~\arhangelskii's from~\cite{arh78}.

\begin{theorem}[\arhangelskii~\cite{arh78}, 1978]\label{arh2.2.4}
Let $X$ be a compactum and let $\kappa=t(X)$. There exists a non-empty $G_\kappa$-set $G$ and a set $H\sse X$ such that $|H|\leq 2^{\kappa}$ and $G\sse\overline{H}$. 
\end{theorem}

Using Theorems \ref{density}, \ref{arh2.2.4}, and \ref{pyt} a simplified proof of de la Vega's Theorem was given in \cite{CR2012}. We give this below as our third fundamental proof.
\begin{theorem} [de la Vega \cite{DLV2006}, 2006]\label{DLV} If $X$ is a homogeneous compactum then $|X|\leq 2^{t(X)}$.
\end{theorem}
\begin{proof}(\cite{CR2012})
Let $\kappa=t(X)$. By Theorem \ref{arh2.2.4} there exists a non-empty $G_\kappa$-set contained in the closure of a set of size at most $\kappa$. Fix a point $p\in G$ and, as in previous proofs, we obtain homeomorphisms $h_x:X\to X$ such that $h_x(p)=x$ for all $x\in X$. $\scr{G}=\{h_x[G]:x\in X\}$ is a cover of $X$ consisting of $G_\kappa$-sets. There exists a family $\scr{H}=\{H_G:G\in\scr{G}\}$ such that $G\sse\cl{H_G}$ and $|H_G|\leq\kappa$ for all $G\in\scr{G}$.

By Pytkeev's Theorem \ref{pyt} there exists $\scr{G}^\prime\sse\scr{G}$ such that $\scr{G}^\prime$ covers $X$ and $|\scr{G}^\prime|\leq 2^\kappa$. It follows that $X=\Un\scr{G}^\prime\sse\Un_{G\in\scr{G}^\prime}\cl{H_G}\sse\cl{\Un_{G\in\scr{G}^\prime}H_G}$.
Thus, $H=\Un_{G\in\scr{G}^\prime}H_G$ is dense in $X$ and $|H|\leq 2^\kappa\cdot\kappa=2^\kappa$. Therefore $d(X)\leq 2^\kappa$. By Theorem~\ref{density} above and \sapirovskii's result that $\pichar{X}\leq t(X)$ for a compact space $X$, we have $|X|\leq d(X)^{\pichar{X}}\leq (2^\kappa)^\kappa=2^\kappa$.
\end{proof}

While much of the work in this proof is done by Theorem~\ref{pyt}, which itself is an elaborate closing-off argument, the homogeneity of the space is not utilized in~\ref{pyt}. Instead the homogeneity is applied in two straightforward and elegant ways: first by using Theorem \ref{arh2.2.4} and homeomorphisms to cover the space by non-empty $G_\kappa$-sets, and second through the use of Theorem~\ref{density}.

The compactness condition is necessary in de la Vega's Theorem. Indeed, it does not hold for all countably compact homogeneous spaces, nor all H-closed homogeneous spaces, as the next example from \cite{CPR2017} shows.

\begin{theorem}[C., Porter, Ridderbos \cite{CPR2017}, 2017]
There exists a countably compact, H-closed, Urysohn, separable, countably tight, homogeneous space $X$ such that $|X|=2^\mathfrak{c}$.
\end{theorem}

\begin{proof}
Let $Y$ be the Cantor Cube $2^\mathfrak{c}$ with it usual topology and let $X$ be the countable tightness modification of $Y$. That is, the closure of a set $A$ in $X$ is given by 
$$cl_X(A)=\Un_{B\in[A]^{\leq\omega}}cl_Y(B).$$
$X$ has a finer topology than $Y$, demonstrating that $X$ is not compact as compact spaces are minimal Hausdorff. However, by Theorem 4.2 in \cite{CPR2017}, $X$ is countably compact, H-closed, countably tight, and separable. Furthermore, since $Y$ is the semiregularization of $X$, $X$ is also Urysohn. (See \cite{por88}).
\end{proof}

De la Vega's Theorem was extended to power homogeneous compacta in~\cite{avr07}.

\begin{theorem}[\arhangelskii, van Mill, and Ridderbos \cite{avr07}, 2007]
If $X$ is a power homogeneous compactum then $|X|\leq 2^{t(X)}$.
\end{theorem}

In 2018 \juhasz~and van Mill introduced new techniques and improved de la Vega's Theorem in the countable case. Considering  compact homogenous spaces that are $\sigma$-CT (a countable union of countably tight subspaces), they obtained the following two results.

\begin{theorem}[\juhasz, van Mill \cite{JVM2018}, 2018]\label{jvm}
If a compactum $X$ is the union of countably many dense countably tight subspaces and $X^\omega$ is homogeneous, then $|X|\leq\cont$.
\end{theorem}
\begin{theorem}[\juhasz, van Mill \cite{JVM2018}, 2018]\label{finite}
If $X$ is an infinite homogeneous compactum that is the union of finitely many countably tight subspaces, then $|X|\leq\cont$.
\end{theorem}

A crucial ingredient in these theorems was a strengthening of \arhangelskii's Theorem~\ref{arh2.2.4} in the countable case. (Recall  Theorem~\ref{arh2.2.4} played a central role in proving de la Vega's Theorem). The strengthening has a deep and sophisticated proof. A subset of a space is \emph{subseparable} if it is contained in the closure of a countable set.

\begin{theorem}[\juhasz, van Mill \cite{JVM2018}, 2018]\label{subseparable}
Every $\sigma$-CT compactum $X$ has a non-empty subseparable $G_\delta$-set.
\end{theorem}

Soon afterwards the ``$X^\omega$ is homogeneous'' condition in Theorem~\ref{jvm} was generalized to ``$X$ is power homogeneous'' in \cite{Carlson2018}.

\begin{theorem}[C. \cite{Carlson2018}, 2018]\label{phctblytight}
If a power homogeneous compactum $X$ is the union of countably many dense countably tight subspaces, then $|X|\leq\cont$.
\end{theorem}

Motivated by the results of \juhasz~and van Mill, the author introduced a cardinal invariant known as $wt(X)$, the weak tightness, in \cite{Carlson2018}. To define it we need the notion of the \emph{$\kappa$-closure} $cl_\kappa A$ of a set $A$ in a space $X$ for a cardinal $\kappa$. This is defined by $cl_\kappa A=\Un\{\overline{B}:B\in [A]^{\leq\kappa}\}$. The \emph{weak tightness} $wt(X)$ of $X$ is defined as the least infinite cardinal $\kappa$ for which there is a cover $\scr{C}$ of $X$ such that $|\scr{C}|\leq 2^\kappa$ and for all $C\in\scr{C}$, $t(C)\leq\kappa$ and $X=cl_{2^\kappa}C$. We say that $X$ is \emph{weakly countably tight} if $wt(X)=\omega$. It is clear that $wt(X)\leq t(X)$. Example 2.3 in \cite{BC2020a} provides a straightforward example of a compact group of tightness $\omega_1$ such that, under $2^{\omega}=2^{\omega_1}$, $X$ is weakly countably tight.

The condition ``$X=cl_{2^\kappa}C$'' in the above definition can be difficult to work with. The next proposition gives additional conditions under which this condition can be relaxed to ``$C$ is dense in $X$''.

\begin{proposition}[\cite{Carlson2018}]\label{wteasy}
Let $X$ be a space, $\kappa$ a cardinal, and $\scr{C}$ a cover of $X$ such that $|\scr{C}|\leq 2^\kappa$, and for all $C\in\scr{C}$, $t(C)\leq\kappa$ and $C$ is dense in $X$. If $t(X)\leq 2^\kappa$ or $\pi_\chi(X)\leq 2^\kappa$ then $wt(X)\leq\kappa$.
\end{proposition}

Pytkeev's Theorem~\ref{pyt} has an improvement using $wt(X)$.

\begin{theorem}[C. \cite{Carlson2018}]\label{pytimprove}
Let $X$ be a compactum and $\kappa$ an infinite cardinal. Then $L(X_\kappa)\leq 2^{wt(X)\cdot\kappa}$.
\end{theorem}

Additionally, Bella and the author were able to give a result that amounts to a variation of both Theorem~\ref{arh2.2.4} and \ref{subseparable}.

\begin{theorem}[Bella, C. \cite{BC2020a}, 2020]\label{variation}
Let $X$ be a compactum and let $\kappa=wt(X)$. Then there exists
a non-empty closed set $G\sse X$ and a $\scr{C}$-saturated set $H\in[X]^{\leq 2^\kappa}$
such that $G\sse\overline{H}$ and $\chi(G,X)\leq\kappa$.
\end{theorem}

Using Theorems~\ref{pytimprove} and \ref{variation}, Bella and the author were able to give a full improvement to de la Vega's Theorem in \cite{BC2020a}. Recall that $\pi\chi(X)\leq t(X)$ for a compactum $X$.

\begin{theorem}[Bella, C. \cite{BC2020a}, 2020]\label{cpthomog}
If $X$ is a homogeneous compactum then $|X|\leq 2^{wt(X)\pi\chi(X)}$.
\end{theorem}

Below we isolate the case of Theorem \ref{cpthomog} where all
cardinal invariants involved are countable. It follows directly
from Proposition~\ref{wteasy} and the above. Compare Corollary~\ref{hcountable} with Theorem~\ref{jvm}.

\begin{corollary}\label{hcountable}
Let $X$ be a homogeneous compactum of countable
$\pi$-character with a cover $\scr{C}$ such that
$|\scr{C}|\leq\cont$ and for all $C\in\scr{C}$, $C$ is countably
tight and dense in $X$. Then $|X|\leq\cont$.
\end{corollary}

Another corollary to Theorem~\ref{cpthomog} follows directly from the fact that in an compact, $T_5$ space there exists a point of countable $\pi$-character. (This is due to \sapirovskii). If the space $X$ is additionally homogeneous then $\pi\chi(X)=\omega$. This corollary has not been previously mentioned in the literature.

\begin{corollary}\label{T5}
If $X$ is compact, $T_5$, and homogeneous, then $|X|\leq 2^{wt(X)}$.
\end{corollary}

\section{Generalizations of de la Vega's Theorem}

This section is devoted to extensions of de la Vega's Theorem; that is, results that directly imply that theorem in a more generalized setting. Natural questions arise, such as, does Lindel\"of suffice instead of the compactness property? The answer to this question is no. In \cite{CR2012}, an example of a $\sigma$-compact, homogeneous space $X$ was constructed with the property $|X|>2^{L(X)\pi\chi(X)t(X)}$. This shows $2^{L(X)t(X)}$ is not a bound for the cardinality of every Hausdorff homogeneous space. 

Exactly what are the necessary properties of compactness needed in this theorem? It turns out that one pair of necessary properties are Lindel\"of and countable point-wise compactness type. The \emph{point-wise compactness type} $pct(X)$ of a space $X$ is the least infinite cardinal $\kappa$ such that $X$ can be covered by compact sets $K$ such that $\chi(K,X)\leq\kappa$. Clearly compact spaces are of countable point-wise compactness type. Also, all locally compact spaces have this property. In \cite{DLVthesis}, de la Vega showed that $|X|\leq 2^{L(X)t(X)pct(X)}$ for any regular homogeneous space, and this bound was shown to be valid for regular power homogeneous spaces in~\cite{Rid06}. In \cite{CR2012}, the regularity property was shown to be unnecessary. 
 
\begin{theorem}[C., Ridderbos \cite{CR2012}, 2012]\label{ltpct}
If $X$ is a power homogeneous Hausdorff space then $|X|\leq 2^{L(X)t(X)pct(X)}$.
\end{theorem}

Thus, for example, the cardinality bound $2^{t(X)}$ holds for all locally compact, Lindel\"of homogeneous Hausdorff spaces.

The next five theorems represent slight improvements of Theorem~\ref{ltpct}. For a cardinal $\kappa$ and a space $X$, a subset $\{x_\alpha:\alpha\leq\kappa\}\sse X$ is a \emph{free sequence of length }$\kappa$ if for every $\beta<\kappa$, $cl_X\{x_\alpha:\alpha<\beta\}\meet cl_X\{x_\alpha:\alpha\geq\beta\}=\es$. The free sequence number $F(X)$ is the supremum of the lengths of all free sequences in $X$. It is well-known that $t(X)=F(X)$ if $X$ is a compactum. In addition, as $F(X)\leq L(X)t(X)$ for any space $X$, the following theorem improves Theorem~\ref{ltpct}.

\begin{theorem}[C., Porter, Ridderbos \cite{CPR2012}, 2012]\label{F(X)}
If $X$ is a power homogeneous Hausdorff space then $|X|\leq 2^{L(X)F(X)pct(X)}$.
\end{theorem}

The invariant $aL_c(X)$, the \emph{almost Lindel\"of degree with respect to closed sets}, is the smallest infinite cardinal $\kappa$ such that for every closed subset $C$ of $X$ and every collection $\scr{U}$ of open sets in $X$ that cover $C$, there is a subcollection $\scr{V}$ of $\scr{U}$ such that $|\scr{V}|\leq\kappa$ and $\{\overline{U}:U\in\scr{V}\}$ covers $C$. It is clear that $aL_c(X)\leq L(X)$.

\begin{theorem}[C., Porter, Ridderbos \cite{CPR2012}, 2012]\label{alc}
If $X$ is a power homogeneous Hausdorff space then $|X|\leq 2^{aL_c(X)t(X)pct(X)}$.
\end{theorem}

Recently in \cite{BC2020b}, the Lindel\"of degree $L(X)$ in Theorem~\ref{ltpct} was replaced by the cardinal invariant $pwL_c(X)$, introduced by Bella and Spadaro in \cite{BS2020}. The \emph{piecewise weak Lindel\"of degree for closed sets} $pwL_c(X)$ of $X$ is the least infinite cardinal $\kappa$ such that for every closed set $F\sse X$, for every open cover $\scr{U}$ of $F$, and every decomposition $\{\scr{U}_i:i\in I\}$ of $\scr{U}$, there are families $\scr{V}_i\in[\scr{U}_i]^{\leq\kappa}$ for every $i\in I$ such that $F\sse\Un\{\overline{\Un\scr{V}_i}:i\in I\}$. It is clear that $pwL_c(X)\leq L(X)$ and, importantly, it can be shown that $pwL_c(X)\leq c(X)$. 
\begin{theorem} [Bella, C. \cite{BC2020b}, 2020]\label{pwlc}
If $X$ is homogeneous and Hausdorff then $|X|\leq 2^{pwL_c(X)t(X)pct(X)}$.
\end{theorem}

While it is clear that Theorem~\ref{pwlc} is an improvement of Theorem~\ref{ltpct}, it is also an improvement of Theorem~\ref{alc} as it can be shown that $pwL_c(X)\leq aL_c(X)$. Furthermore, it is a variation of Theorem~\ref{ER}; that is, $2^{c(X)\pi\chi(X)}$ is a bound for the cardinality of every homogeneous Hausdorff space. This is because $pwL_c(X)\leq c(X)$ and $\pi\chi(X)\leq t(X)pct(X)$ for Hausdorff spaces.

In \cite{BC2020b}, a consistent improvement of Theorem~\ref{ltpct} was given using the linearly Lindel\"of degree $lL(X)$. A space $X$ is \emph{linearly Lindel\"of} provided that every increasing open cover of $X$ has a countable subcover. More generally, we define the \emph{linear Lindel\"of degree} $lL(X)$ of $X$ as the smallest cardinal $\kappa$ such that every increasing open cover of $X$ has a subcover of size not exceeding $\kappa$. Equivalently, $lL(X)\leq\kappa$ if every open cover of $X$ has a subcover $\scr{U}$ such that $|\scr{U}|$ has cofinality at most $\kappa$. 

\begin{theorem}[Bella, C. \cite{BC2020b}, 2020]\label{lL}
Assume $2^\kappa<\kappa^{+\omega}$ or $2^{<2^\kappa}=2^\kappa$. If $X$ is Hausdorff, homogeneous, and $\kappa=lL(X)F(X)pct(X)$, then $|X|\leq 2^\kappa$.
\end{theorem}

Our last improvement of Theorem~\ref{ltpct} gives a bound for the cardinality of an open set in a power homogeneous space.

\begin{theorem}[Bella, C. \cite{BC}, 2018]\label{openset}
If $X$ is a power homogeneous Hausdorff space and $U\sse X$ is an open set, then $|U|\leq 2^{L(\overline{U})t(X)pct(X)}$
\end{theorem}

The next four results from \cite{SS2018}, \cite{BC}, and \cite{BCG2020} represent extensions of de la Vega's Theorem in a different direction using the invariants $wL(X)$ or $wL_c(X)$. The \emph{weak Lindel\"of degree} of a space $X$ is the least infinite cardinal $\kappa$ such that every open cover $\scr{U}$ of $X$ has a subfamily $\scr{V}$ such that $|\scr{V}|\leq\kappa$ and $X=\overline{\Un\scr{V}}$. The invariant $wL_c(X)$, the \emph{weak Lindel\"of degree with respect to closed sets}, is the smallest infinite cardinal $\kappa$ such that for every closed subset $C$ of $X$ and every collection $\scr{U}$ of open sets in $X$ that cover $C$, there is a subcollection $\scr{V}$ of $\scr{U}$ such that $|\scr{V}|\leq\kappa$ and $C\sse \overline{\Un\scr{V}}$. It is clear that $wL(X)\leq wL_c(X)\leq aL_c(X)$.

\begin{theorem}[Spadaro, Szeptycki \cite{SS2018}, 2018]\label{initial}
If $X$ is an initially $\kappa$-compact power homogeneous $T_3$ space then $|X|\leq 2^{F(X)wL_c(X)}$.
\end{theorem}

\begin{theorem}[Bella, C. \cite{BC}, 2018]\label{piLindelof}
If $X$ is a regular power homogeneous space and with a $\pi$-base $\scr{B}$ such that $\overline{B}$ is Lindel\"of for all $B\in\scr{B}$, then $|X|\leq 2^{wL(X)t(X)pct(X)}$.
\end{theorem}

As locally compact spaces satisfy the hypotheses in Theorem~\ref{piLindelof}, we have the following corollary.

\begin{corollary}[Bella, C. \cite{BC}, 2018]\label{lcphup}
If $X$ is a locally compact power homogeneous space then $|X|\leq 2^{wL(X)t(X)}$.
\end{corollary}
The above theorem indicates that the compactness condition in de la Vega's Theorem can be replaced with another pair of conditions: locally compact and weakly Lindel\"of. It turns out that Corollary~\ref{lcphup} can be given an improved conclusion. This was demonstrated in \cite{BCG2020}.

\begin{theorem}[Bella, C., Gotchev \cite{BCG2020}, 2020]\label{lcphdown}
If $X$ is a locally compact power homogeneous space then $|X|\leq wL(X)^{t(X)}$.
\end{theorem}

\section{Other results}
Recently it was shown in \cite{BCG2020} that if $X$ is an extremally disconnected space then $c(X)\leq w(X)\pi\chi(X)$. Using Theorem~\ref{ERPH}, the following is an immediate consequence.

\begin{theorem}[Bella, C., Gotchev \cite{BCG2020}, 2020]\label{ed}
If $X$ is power homogeneous and extremally disconnected then $|X|\leq 2^{wL(X)\pi\chi(X)}$.
\end{theorem}

One should regard the bound in Theorem~\ref{ed} as being ``small'' as $wL(X)$ and $\pi\chi(X)$ are generally thought of as small cardinal invariants. Observe that it follows from Theorem~\ref{ed} that an H-closed, extremally disconnected, power homogeneous space has cardinality at most $2^{\pi\chi(X)}$. However, it was shown in \cite{car07} that an infinite H-closed 
extremally disconnected space cannot be power homogeneous. This latter result is an extension of a result of 
Kunen~\cite{K90} that an infinite compact F-space is not power homogeneous.

Given a space $X$, the \emph{diagonal} of $X^2$, denoted by $\Delta_X$, is the set $\{(x,x):x\in X\}$. $X$ is said to have a \emph{regular $G_\delta$-diagonal} if there exists a countable family $\scr{U}$ of open sets in $X^2$ such that $\Delta_X=\Meet\scr{U}=\Meet\{\overline{U}:U\in\scr{U}\}$. A cardinality bound for homogeneous spaces with a regular $G_\delta$-diagonal was given in \cite{BBR2014}.

\begin{theorem}[D. Basile, Bella, Ridderbos \cite{BBR2014}, 2014]\label{regdiag}
If $X$ is a homogeneous space with a regular $G_\delta$-diagonal, then $|X|\leq wL(X)^{\pi\chi(X)}$.
\end{theorem}

A notion related to homogeneity is known as countable dense homogeneity. A separable space $X$ is \emph{countable dense homogeneous} (CDH) if given any two countable dense subsets $D$ and $E$ of $X$, there is a homeomorphism $h:X\to X$ such that $h[D]=E$. Separability is included in the definition as clearly this notion is of interest only if $X$ has a countable dense subset. Not every CDH space is homogeneous, however every connected CDH space is homogeneous \cite{FL1987}.

In \cite{av14b} it was shown that the cardinality of a CDH space is as most $\cont$.

\begin{theorem}[\arhangelskii, van Mill \cite{av14b}, 2014]\label{CDH}
The cardinality of a CDH space is at most $\cont$.
\end{theorem}

\section{Questions and a Table of Bounds}
Recall that in Theorem~\ref{CRH} it was shown that $|H(X)|\leq 2^{c(X)\pi\chi(X)sd(X)}$ for a Hausdorff space $X$. In addition, in Theorem~\ref{ER} it was shown that $|X|\leq 2^{c(X)\pi\chi(X)}$ if $X$ is Hausdorff and homogeneous. In light of these theorems, the following was asked by the author and Ridderbos in \cite{CR08}. 
\begin{question}[C., Ridderbos, 2008]
Is $|H(X)|\leq 2^{c(X)\pi\chi(X)}$ for a Hausdorff space $X$?
\end{question}

As it was shown in \cite{BC2020a} that the cardinality of a homogeneous compactum is at most $2^{wt(X)\pi\chi(X)}$ (Theorem~\ref{cpthomog}), it is natural to ask if either of the two cardinal invariants $wt(X)$ and $\pi\chi(X)$ can be removed from this bound. The next two questions were asked in \cite{BC2020a}. The second was additionally asked by de la Vega in \cite{DLVthesis}. These two questions appear to be quite challenging. Answering either in the affirmative would likely require new breakthrough techniques, while counter-examples would likely be complicated and have intriguing properties.
\begin{question}[Bella, C. \cite{BC2020a}, 2020]
Is the cardinality of a homogeneous compactum at most $2^{wt(X)}$?
\end{question}
\begin{question}[de la Vega \cite{DLVthesis}, 2005, Bella, C. \cite{BC2020a}, 2020]
Is the cardinality of a homogeneous compactum at most $2^{\pi\chi(X)}$?
\end{question}
The power homogeneous case of Theorem~\ref{cpthomog} is also an open question.
\begin{question}[Bella, C. \cite{BC2020a}, 2020]
Is the cardinality of a power homogeneous compactum at most $2^{wt(X)\pi\chi(X)}$?
\end{question}
By results of \sapirovskii, every $T_5$ compactum has a point of countable $\pi$-character. It follows from Theorem~\ref{ER} that the cardinality of a homogeneous $T_5$ compactum is at most $2^{c(X)}$. This was observed by van Mill in \cite{VM05} and was proved for power homogeneous $T_5$ compacta in \cite{Rid09}. Note additionally that Corollary~\ref{T5} states that the cardinality of a homogeneous $T_5$ compactum is at most $2^{wt(X)}$. Van Mill asked if the cardinality of such spaces is in fact at most $\mathfrak{c}$.

\begin{question}[van Mill \cite{VM05}, 2005] 
Is the cardinality of every $T_5$  homogeneous compactum at most $\mathfrak{c}$?
\end{question}

In light of the various cardinality bounds for homogeneous-like spaces using the weak Lindel\"of degree $wL(X)$, the following was asked in \cite{BC2018}.

\begin{question}
If $X$ is power homogeneous and Tychonoff, is $|X|\leq 2^{wL(X)t(X)pct(X)}$?
\end{question}

\newpage

\begin{table}[h!]
  \begin{center}
    \caption{Strongest known cardinality bounds on spaces with homogeneous-like properties.}
    \label{tab:table1}
    \begin{tabular}{l|c|c|c|r} 
      \textbf{Bound on $|X|$} & \textbf{Hypotheses on $X$} & \textbf{Proved in} & \textbf{Year} & \textbf{Thm}\\
      \hline
      $|RO(X)|^{q\psi(X)}$ & homog., Hausdorff & Ismail \cite{ism81} & 1981 & \ref{ismailq}\\
       \hline
      $d(X)^{\pi n\chi(X)}$ & homog., Hausdorff & (current paper) & 2020 & \ref{densityn}\\
       \hline
      $d(X)^{\pi\chi(X)}$ & power homog., Hausdorff & Ridderbos \cite{Rid06}& 2006 & \ref{ridHausdorff}\\
       \hline
      $d_\theta(X)^{\pi\chi(X)}$ & power homog., Urysohn & C. \cite{car07} & 2007 & \ref{Ury}\\
       \hline
      $\pi\chi(X)^{c(X)q\psi(X)}$ & homog., $T_3$ & (current paper) & 2020 & \ref{regular}\\
       \hline
      $2^{c(X)\pi\chi(X)}$ & power homog., Hausdorff & C., Ridderbos \cite{CR08} & 2008 & \ref{ERPH}\\
       \hline
      $2^{Uc(X)\pi\chi(X)}$ & power homog., & Bonanzinga, C., & 2018 & \ref{phUry}\\
                                        & (Urysohn or quasiregular) & Cuzzup\'e, Stavrova \cite{BCCS2018} &  & \\
       \hline
      $\mathfrak{c} $ & homog. compactum that & \juhasz, van Mill \cite{JVM2018}& 2018 & \ref{finite}\\
                               & is the union of finitely many &  &  & \\
                                 & countably tight subspaces &  &  & \\
       \hline
      $\mathfrak{c} $ & power homog. compactum & C. \cite{Carlson2018}& 2018 & \ref{phctblytight}\\
                               & that is the union of  &  &  & \\
                               & countably many dense &  &  & \\
                               & countably tight subspaces &  &  & \\
       \hline
      $2^{wt(X)}$ & homog., compact, $T_5$ & (current paper) & 2020 & \ref{T5}\\
       \hline
      $2^{wt(X)\pi\chi(X)}$ & homog., compactum & Bella, C. \cite{BC2020a} & 2020 & \ref{cpthomog}\\
       \hline
      $2^{L(X)F(X)pct(X)}$ & power homog., Hausdorff & C., Porter, & 2012 & \ref{F(X)}\\
     					& & Ridderbos \cite{CPR2012} & & \\
       \hline
      $2^{pwL_c(X)t(X)pct(X)}$ & homog., Hausdorff & Bella, C., \cite{BC2020b} & 2020 & \ref{pwlc}\\
       \hline
      $2^{lL(X)F(X)pct(X)}$ & homog., Hausdorff & Bella, C., \cite{BC2020b} & 2020 & \ref{lL}\\
      					& ($2^\kappa<\kappa^{+\omega},$ or $2^{<2^\kappa}=2^\kappa$ & & & \\
					& $\kappa= lL(X)F(X)pct(X))$ & & & \\
       \hline
      $2^{F(X)wL_c(X)}$ & power homog., $T_3$, & Spadaro, & 2018 & \ref{initial}\\
                              & initially $\kappa$-compact  & Szeptycki \cite{SS2018}  &  & \\
       \hline
  
      $2^{wL(X)t(X)pct(X)}$ & power homog., $T_3$,  & Bella, C., \cite{BC} & 2018 & \ref{piLindelof}\\
      					& $\pi$-base $\scr{B}$ such that $\overline{B}$& & & \\
					& is Lindel\"of for all $B\in\scr{B}$ & & & \\
       \hline
      $wL(X)^{t(X)}$ & power homog., loc. compact & Bella, C., Gotchev \cite{BCG2020} & 2020 & \ref{lcphdown}\\
       \hline
      $2^{wL(X)\pi\chi(X)}$ & power homog.,  & Bella, C., Gotchev \cite{BCG2020} & 2020 & \ref{ed}\\
      					& extremally disconnected & & & \\
       \hline
      $wL(X)^{\pi\chi(X)}$ & homog,  & D. Basile, Bella, & 2014 & \ref{regdiag}\\
      					& regular $G_\delta$ diagonal &Ridderbos \cite{BBR2014} & & \\
       \hline
      $\mathfrak{c} $ & countable dense homog. & \arhangelskii, & 2014 & \ref{CDH}\\
      					& separable &van Mill \cite{av14b} & & \\
       \hline
    \end{tabular}
  \end{center}
\end{table}

\end{document}